\documentclass[reqno]{amsart}
\usepackage{fullpage}
\usepackage{array}
\usepackage{mathrsfs}
\usepackage{hyperref}
\usepackage{graphicx,color}
\usepackage{float}

\newcommand{\varied}

\newcommand{\black}{\color{black}}

\newcommand{\spam}{\ensuremath{\textnormal{span}}}

\newtheorem{theorem}{Theorem}[section]

\newtheorem{lemma}[theorem]{Lemma}
\newtheorem{corollary}[theorem]{Corollary}
\newtheorem{remark}[theorem]{Remark}
\newtheorem{proposition}[theorem]{Proposition}

\newcolumntype{C}[1]{>{\centering\let\newline\\\arraybackslash\hspace{0pt}}m{#1}}

\usepackage[utf8]{inputenc}
\usepackage{amsfonts}
\usepackage{amssymb}
\usepackage{verbatim}

\numberwithin{equation}{section}

\begin{document}
\title{Group graded algebras and varieties with quadratic codimension growth}

\author{Wesley Quaresma Cota$^1$}
\thanks{{\it E-mail addresses:} quaresmawesley@gmail.com}

\thanks{\footnotesize $^{1}$ Partially supported by FAPESP, grant no.~2025/05699-0.}



\dedicatory{Instituto de Matemática e Estatística, Universidade de São Paulo, São Paulo,
Brazil}

\begin{abstract}  
Let $A$ be an associative algebra graded by a finite group $G$ over a field ${F}$ of characteristic zero. One associates to $A$ the sequence of $G$-graded codimensions $c_n^G(A)$, $n=1,2,\ldots$, which measures the growth of the polynomial identities satisfied by $A$. It is known that this sequence is either polynomially bounded or grows exponentially. In this paper, we study unitary $G$-graded varieties of polynomial codimension growth. In particular, we classify the varieties generated by unitary algebras with quadratic codimension growth and show that these varieties can be described as a direct sums of algebras that generate minimal $G$-graded varieties.
\end{abstract}

\maketitle

\section{Introduction}

In recent years, the theory of polynomial identities has undergone remarkable development, with special attention devoted to the study of numerical invariants associated with a given algebra $A$. Among these invariants, the codimension sequence $c_n(A),\, n = 1, 2, \ldots,$ plays a fundamental role. By definition, $c_n(A)$ is the dimension of the space $P_n$ of multilinear polynomials in $n$ variables modulo the $T$-ideal $\mathrm{Id}(A)$ of polynomial identities of $A$. Since over a field of characteristic zero, every polynomial identity follows from finitely many multilinear ones, the codimension sequence provides a natural and powerful tool for analyzing the asymptotic behavior of the identities satisfied by $A$.

A seminal contribution in this direction is due to Regev~\cite{RG}, who showed that if $A$ is a PI-algebra, then the sequence $\{c_n(A)\}_{n\geq 1}$ is exponentially bounded. Later, Kemer \cite{Kem} showed that the codimension sequence of any PI-algebra either grows exponentially or is polynomially bounded, thus proving that intermediate growth is impossible. This dichotomy laid the foundation for a systematic study of varieties of polynomial growth. Recall that the variety generated by $A$, denoted $\mathrm{var}(A)$, is the class of all algebras $B$ such that $\mathrm{Id}(A) \subseteq \mathrm{Id}(B)$. We say that $\mathrm{var}(A)$ has polynomial growth (of the codimension sequence) if there exist constants $\alpha, t >0$ such that $c_n(A) \leq \alpha n^t$ for all $n \geq 1$.

In this context, it was proved in~\cite{Dren} that whenever $\mathrm{var}(A)$ has polynomial growth, the sequence $c_n(A)$ behaves asymptotically as  
\[
c_n(A) = q n^k + \mathcal{O}(n^{k-1}) \approx q n^k, \quad n \to \infty,
\]  
for some rational constant $q$ and positive integer $k$. Moreover, when $A$ is unitary, the leading coefficient $q$ admits explicit bounds
\[
\frac{1}{k!} \leq q \leq \sum_{j=2}^k \frac{(-1)^j}{j!} \longrightarrow \frac{1}{e},\quad n \to \infty,
\] where $e$ denotes Euler’s number (see~\cite{DrenReg}).

One of the central problems in PI-theory is the classification of varieties according to the growth rate of their codimensions. Despite the relevance of this problem, progress in this direction has been limited. In the ordinary case, unitary varieties with polynomial growth $n^k$ have been classified only for $k \leq 4$ (see~\cite{Petro, Mara3}), while in the general setting the classification is known only for varieties of at most linear growth~\cite{GLa} and for the minimal varieties with quadratic growth~\cite{Sandra}.

A further layer of complexity arises when algebras are endowed with additional structures, such as $G$-grading. Group gradings and their associated graded identities have become indispensable tools in PI-theory, since ordinary identities can be seen as a particular case of $G$-graded identities. Within this framework, it is natural to extend the notions of codimensions and polynomial growth to the graded setting. In this paper, we focus on associative PI-algebras graded by a finite group $G$, whose $G$-graded codimension sequences exhibit polynomial growth.

Compared to the ordinary case, the classification of $G$-graded varieties with polynomial growth is still in its infancy. The only general result available concerns the classification of varieties with at most linear growth (see \cite{Plamen}). An important exception arises when $G \cong \mathbb{Z}_2$: in this case, the classification of unitary $\mathbb{Z}_2$-graded varieties with quadratic growth was established in~\cite{Dafne}. More recently, in \cite{Cota2} the author completed the classification of minimal varieties with quadratic codimension growth. Recall that a $G$-graded variety $\mathrm{var}^G(A)$ is minimal of polynomial growth $n^k$ if $c_n^G(A)\approx \alpha n^k$ and, for every proper subvariety $\mathcal{U}\subsetneq \mathrm{var}^G(A)$, we have $c_n^G(\mathcal{U})\approx \beta n^t$ for some $t<k$. Minimal varieties are of particular importance, not only for their intrinsic structure, but also because they often appear as fundamental building blocks in the construction of more general varieties.

In this work, we address the classification of unitary $G$-graded varieties with quadratic codimension growth. Our main results show that, depending on the group $G$, there may exist an infinite family of such varieties. Nevertheless, this family is well structured: all varieties of this type can be explicitly described within a parametrized class of algebras. A key consequence of our classification is that every unitary $G$-graded variety with quadratic growth decomposes as a direct sum of minimal varieties, providing a precise structural description of this class.

\section{Group graded algebras}

Throughout the paper, let $F$ denote a field of characteristic zero, $G=\{g_1=1, g_2, \ldots , g_k\}$ a finite multiplicative group and $A$ an associative algebra over $F$. Recall that an algebra $A$ is called a $G$-graded (or group graded) algebra if it admits a decomposition $
A=\bigoplus_{g\in G} A_g$ into a direct sum of subspaces satisfying the condition that $A_g A_h \subseteq A_{gh}$ for all $g,h \in G$. A nonzero element $a\in A_g$ is called {homogeneous of degree} $g$.

Let $F\langle X\rangle$ denote the free associative algebra on a countable set $X=\{x_1,x_2,\ldots\}$ over $F$. We decompose $X$ as
$
X=\bigcup_{g\in G} X_{g},$ where $X_{g}=\{x_{1,g},x_{2,g},\ldots\}, \, g\in G$, is a disjoint set of variables of homogeneous degree $g$. The degree of a monomial $x_{i_1,g_{j_1}}\cdots x_{i_t,g_{j_t}}$ is given by $g_{j_1}\cdots g_{j_t}$. Denote by $\mathcal{F}_{g}$ the subspace of $F\langle X\rangle$ spanned by all monomials of homogeneous degree $g$. Hence,
\[
F\langle X\rangle=\bigoplus_{g\in G} \mathcal{F}_{g},
\]
is a $G$-grading on $F\langle X\rangle$. This $G$-graded algebra is called the free $G$-graded algebra of countable rank over $F$, denoted by $F\langle X,G\rangle$. The elements $f\in F\langle X,G\rangle$ are called $G$-graded polynomials, or simply polynomials since the context is clear.

An {admissible evaluation} of \( f \) on \( A \) is a substitution of each variable 
\( x_{i,g} \) by an element \( a \in A_g \). A polynomial \( f \in F\langle X \mid G \rangle \) is called a {$G$-graded identity} of \( A \)  if it vanishes under every admissible evaluation on \( A \). In this case, we write \( f \equiv 0 \) in \( A \). The set of all $G$-graded identities of $A$ is denoted by
\[
\textnormal{Id}^G(A)=\{f\in F\langle X,G\rangle \mid f\equiv 0 \text{ on } A\},
\]
which forms a $T_G$-ideal of $F\langle X,G\rangle$, i.e., an ideal invariant under all $G$-graded endomorphisms of the free algebra. It is well known that, in characteristic zero, $\textnormal{Id}^G(A)$ is completely determined by its multilinear elements. For every $n\geq 1$, define
\[
P_n^G=\spam_F\{\,x_{\sigma(1),g_{i_1}}\cdots x_{\sigma(n),g_{i_n}} \;\mid\; \sigma\in S_n,\; g_{i_1},\ldots,g_{i_n}\in G\,\},
\]
the space of multilinear $G$-graded polynomials of degree $n$. For $n\geq 1$, the $n$-th $G$-graded codimension of $A$ is defined as
\[
c_n^G(A)=\dim_F P_n(A),\mbox{ where }P_n(A)=\frac{P_n^G}{P_n^G\cap \operatorname{Id}^G(A)}.
\]
Note that $c_n^G(A)\leq \dim P_n^G = |G|^n n!$. It turns out that, for every $G$-graded algebra $A$, one has $c_n(A) \leq c_n^G(A)$. Moreover, if $A$ is a PI-algebra, the following upper bound holds:
\[
c_n^G(A) \leq |G|^n \, c_n(A),
\]
see \cite{GiamReg}. Therefore, for $G$-graded algebras satisfying an ordinary polynomial identity, the sequence of $G$-graded codimensions is exponentially bounded.

Denote by $\mathcal{V} = \operatorname{var}^G(A)$ the $G$-graded variety generated by $A$, that is, the class of all $G$-graded algebras $B$ such that $\operatorname{Id}^G(A) \subseteq \operatorname{Id}^G(B)$. If $\mathcal{V}$ is generated by a unitary algebra, we say that $\mathcal{V}$ is a unitary variety. We also define $c_n^G(\mathcal{V}) = c_n^G(A)$. Moreover, if $\textnormal{Id}^G(A)=\textnormal{Id}^G(B)$ then we say that $A$ and $B$ are $T_G$-equivalent and we denote $A\sim_{T_G}B$ in this case.

As a consequence of Theorem $9$ in \cite{Valenti}, the sequence of $G$-graded codimensions of $A$ either grows exponentially or is polynomially bounded. Here, we are interested in varieties of polynomial growth, that is, those for which $c_n^G(A) \leq \alpha n^t$, for some constant $\alpha$ and $t$. In this case, La Mattina \cite{La} presented the following characterization, which will be useful in this work.

\begin{theorem} \label{35}
Let $A$ be a $G$-graded algebra over a field $F$. Then $c_n^G(A)$ is polynomially bounded if and only if $A \sim_{T_G} B$, where $B = B_1 \oplus \cdots \oplus B_m,$ where $B_1, \ldots , B_m$ are finite-dimensional $G$-graded algebras over $F$ such that $\dim_F \big( B_i / J(B_i) \big) \leq 1,$ where $J(B_i)$ denotes the Jacobson radical of $B_i$, for all $i = 1, \ldots , m$.
\end{theorem}

The subspaces of proper polynomials are of fundamental importance in the study of unitary $G$-graded algebras. From now on, we assume that $A$ is a $G$-graded algebra with $1$ and we fix an ordering $G=\{g_1=1, g_2, \ldots , g_k\}$ of $G$. Recall that the commutator of length $n$ is defined inductively by $[x_{1}, \ldots , x_{n}] = [[x_{1}, \ldots , x_{n-1}], x_{n}],$ where $[x_{1},x_{2}] = x_{1}x_{2} - x_{2}x_{1}$. A polynomial $f\in P_n^G$ is called a proper $G$-graded polynomial if it is a linear combination of elements of the form
\[
x_{i_1,g_2}\cdots x_{i_p,g_2}\cdots x_{j_1,g_s}\cdots x_{j_\ell,g_s}\, w_1\cdots w_m,
\]
where $w_1,\ldots,w_m$ are left-normed (long) Lie commutators in graded variables, and variables of degree $1_G=g_1$ appear only within commutators. If $A$ is unitary, then $\operatorname{Id}^G(A)$ is generated by its multilinear proper graded polynomials (see \cite{Daniela} and also 
\cite[Proposition~4.3.3]{Drensky} for the ordinary case). 

We denote by $\Gamma_n^G$ the subspace of $P_n^G$ spanned by proper graded polynomials, and set $\Gamma_0^G=\spam\{1\}$. The dimension of the space $\Gamma_n^G$ was computed in \cite[Lemma 2.1]{Daniela} as follows $$\dim \Gamma_n^G=n!\sum_{i=0}^n |G|^{n-i}\frac{(-1)^i}{i!}.$$
The sequence of {proper $G$-graded codimensions} is defined as
\[
\gamma_n^G(A)=\dim\left(\frac{\Gamma_n^G}{\Gamma_n^G\cap \textnormal{Id}^G(A)}\right), \qquad n=0,1,2,\ldots
\]

For a unitary algebra $A$, the connection between the $G$-graded codimensions and proper $G$-graded codimensions (see, for instance, \cite{DrenGiam}) is expressed by
\begin{equation} \label{codecodcentral}
 c_n^G(A)=\sum_{i=0}^n \binom{n}{i}\gamma_i^G(A), \qquad n=0,1,2,\ldots   
\end{equation}

In particular, if the sequence $\{c_n^G(A)\}_{n\geq 0}$ is polynomially bounded, then there exists an integer $t$ such that $\gamma_m^G(A)=0$ for all $m>t$. Consequently, we have the following. 

\begin{theorem} \cite{Daniela} \label{codproper}
    If $A$ is a $G$-graded algebra with $1$ with polynomial growth $n^t$ then
$$c_n^G(A) =\sum_{i=0}^t \binom{n}{i}\gamma_i^G(A)= q n^t + q_1 n^{t-1} + \cdots$$ is a polynomial with rational coefficients. 
Moreover its leading term $q=\frac{\gamma_t^G(A)}{t!}$ satisfies the inequalities $\frac{1}{t!}\leq q\leq \sum_{i=0}^t |G|^{\,t-i}\frac{(-1)^i}{i!}.$
\end{theorem}

\section{\texorpdfstring{Proper $G$-graded cocharacters}{Proper G-graded cocharacters}}

In this section we investigate the spaces of proper $G$-graded polynomials via the representation theory of the symmetric group $S_n$. We focus on their decomposition as $S_{n_1}\times \cdots \times S_{n_k}$-modules and on the corresponding proper cocharacters for a $G$-graded algebra $A$, since the multiplicities in these decompositions determine the behavior of the $G$-graded codimensions.

Recall that $G = \{g_1 = 1, g_2, \ldots, g_k\}$ is a finite group of order $k$, and consider $n = n_1 + \cdots + n_k$, a sum of $k$ nonnegative integers, denoted by $(n_1, \ldots, n_k)$ or simply by $n_1, \ldots, n_k$. Let $P_{n_1, \ldots , n_k}$ be the vector space of multilinear $G$-graded polynomials in the first $n_1$ variables of homogeneous degree $g_1=1$, the second $n_2$ variables of homogeneous degree $g_2$, and so on until the last $n_k$ variables of homogeneous degree $g_k$. Denote by $\Gamma_{n_1, \ldots , n_k}$ the subspace of $P_{n_1, \ldots , n_k}$ consisting of the proper $G$-graded polynomials. Note that the space $\Gamma_n^G$ can be decomposed as follows:
\[
\Gamma_n^G\cong  \bigoplus_{n=n_1+\cdots +n_k}  \binom{n}{n_1, \ldots, n_k}\Gamma_{n_1, \ldots , n_k},
\]
where $\binom{n}{n_1, \ldots, n_{k}}= \frac{n!}{n_1!\cdots n_{k}!}$ denotes the multinomial coefficient. 

For convenience, instead of writing the $k$-tuple $(n_1, \ldots, n_k)$, we shall denote it by $((n_{i_1})_{g_{i_1}}, \ldots, (n_{i_t})_{g_{i_t}})$, where $i_1<\cdots < i_t$ and the zero entries are omitted. For instance, if $n = n_1 + n_2 + n_3$ with $n_1 = 2$, $n_2 = 0$, and $n_3 = 1$, then the $3$-tuple $(2, 0, 1)$ will be denoted by $(2_{1}, 1_{g_3})$. In this notation, the space $\Gamma_{2,0,1}$ is denoted by $\Gamma_{(2_{1}, 1_{g_3})}.$

In the following, we present the decomposition of $\Gamma_n^G$, when $n=1$ and $n=2$, into homogeneous proper subspaces: 
\[
\Gamma_1^G\cong \bigoplus_{g\neq 1} \Gamma_{(1_g)} \quad \mbox{ and }\quad 
\Gamma_2^G \cong  \Gamma_{(2_1)}\oplus \left(\bigoplus_{g\neq 1} \Gamma_{(2_g)}\right)\oplus \left(\bigoplus_{g\neq 1} 2\Gamma_{(1_1,1_g)} \right)\oplus \left(\bigoplus_{g\neq h\neq 1} 2\Gamma_{(1_g,1_h)}\right).
\]

For a decomposition $n=n_1+\cdots + n_k$, we consider the space
\[
\Gamma_{n_1, \ldots ,n_{k}}(A)= \frac{\Gamma_{n_1, \ldots ,n_{k}}}{\Gamma_{n_1, \ldots ,n_{k}}\cap \textnormal{Id}^G(A)}
\]
and define $\gamma_{n_1, \ldots ,n_{k}}(A)=\dim_F  \Gamma_{n_1, \ldots ,n_{k}}(A)$, the proper $(n_1, \ldots , n_k)$-codimension. Consequently, we obtain
\begin{equation} \label{codprop}
	\gamma_n^{G}(A)= \sum_{ n =n_1+\cdots +n_{k} } \binom{n}{n_1,\ldots, n_{k} } \gamma_{n_1,\ldots, n_{k}}(A).
\end{equation}

Recall that there is a natural left action of $S_{n_1, \ldots , n_{k} }:= S_{n_1} \times \cdots \times S_{n_{k}}$ on $P_{n_1, \ldots , n_{k} }$, where the symmetric group $S_{n_i}$ acts by permuting the corresponding variables of homogeneous degree $g_i$, $1\leq i \leq k$. The subspace $\Gamma_{n_1, \ldots , n_{k}}$ is an $S_{n_1, \ldots , n_{k}} $-submodule of $P_{n_1, \ldots , n_{k} }$, and so we can consider its character $\chi(\Gamma_{n_1, \ldots , n_{k}})$. Since $\Gamma_{n_1, \ldots , n_{k}}\cap \textnormal{Id}^G(A)$ is invariant under this action, the space $\Gamma_{n_1, \ldots , n_{k}}(A)$ inherits a structure of left $S_{n_1, \ldots , n_{k} }$-module. We denote its character by $\pi_{n_1, \ldots , n_k}(A)$, which is called the proper $(n_1, \ldots , n_k)$-cocharacter of $A$. 

It is well known that the irreducible characters of $S_{n_1, \ldots, n_k}$ are outer tensor products of irreducible characters of symmetric groups $S_{n_i}$. Each of these irreducible $S_{n_i}$-characters corresponds uniquely to a partition $\lambda_i \vdash n_i$. Therefore, we denote by $\chi_{\lambda_1} \otimes \cdots \otimes \chi_{\lambda_{k}}$ the irreducible $S_{n_1, \ldots, n_k}$-character associated with the multipartition $(\lambda_1, \ldots, \lambda_{k}) \vdash (n_1, \ldots, n_k)$,
where each $\chi_{\lambda_i}$ is the irreducible character of $S_{n_i}$ corresponding to the partition $\lambda_i\vdash n_i$. By complete reducibility, we may decompose $\pi_{n_1, \ldots , n_k}(A)$ into irreducible characters as follows:
\begin{equation} \label{propercharac}
\pi_{n_1, \ldots ,n_k}(A)=\underset{( \lambda_1, \ldots, 
\lambda_k) \vdash (n_1, \ldots ,n_k)}{\sum} {m}_{\lambda_1, \ldots, \lambda_k} \chi_{\lambda_1} \otimes \cdots \otimes \chi_{\lambda_{k}}.    
\end{equation} The degree of the irreducible $S_{n_1, \ldots , n_{k} }$-character $\chi_{\lambda_1} \otimes \cdots \otimes \chi_{\lambda_{k}}$ is given by $d_{\lambda_1} 
\cdots d_{\lambda_k}$, where $d_{\lambda_i}$ is the degree of the irreducible character $\chi_{\lambda_i}$ given by the Hook Formula. 

There is a well-established method for computing the multiplicities in the decomposition of the proper $(n_1, \ldots ,n_k)$-cocharacters, which relies on the representation theory of $GL_m$-modules, where $GL_m$ is the general linear group. For a more detailed description, we recommend \cite[Section 12.4]{Drensky} and the paper \cite[Section 4]{Tatiana} for the case where $G\cong \mathbb{Z}_2$. 

The multiplicity $m_{\lambda_1, \ldots, \lambda_k}$ is determined by the maximal number of linearly independent proper highest weight vectors (h.w.v.s) $f_{\lambda_1, 
\ldots , \lambda_k}$ corresponding to a multipartition $(\lambda_1, \ldots, \lambda_k)$ of $(n_1, \ldots, n_k)$. The construction of such proper h.w.v.s extends in a natural way from the $\mathbb{Z}_2$-graded case treated in \cite{Tatiana}.

In order to compute the multiplicities in the decomposition \eqref{propercharac}, we only need the following result. 

\begin{proposition} \cite{Tatiana} \label{propos}
    Let $A$ be a unitary $G$-graded algebra with proper $(n_1, \ldots, n_k)$-cocharacter given in \eqref{propercharac}. Then, $m_{\lambda_1, \ldots, \lambda_k}\neq 0$ if and only if there exists a proper h.w.v. $f_{\lambda_1, \ldots, \lambda_k}$ associated to $(\lambda_1, \ldots, \lambda_k)
    \vdash (n_1, \ldots, n_k)$ such that $f_{\lambda_1, \ldots, \lambda_k}\notin \textnormal{Id}^G(A)$. Moreover, $m_{\lambda_1, \ldots, \lambda_k}$ is equal to the maximal number of proper h.w.v.s associated to $(\lambda_1, \ldots, \lambda_k)
    \vdash (n_1, \ldots, n_k)$ which are linearly independent modulo $\textnormal{Id}^G(A)$.
\end{proposition}

\begin{remark} \label{multibounded}
Observe that if $A$ has a proper $(n_1, \ldots, n_k)$-cocharacter as in (\ref{propercharac}) and $B$ is a unitary $G$-graded algebra with $(n_1, \ldots, n_k)$-cocharacter  
\[
\pi_{n_1, \ldots ,n_k}(B) = \sum_{(\lambda_1, \ldots, \lambda_k)\,\vdash\, (n_1, \ldots ,n_k)} 
\widetilde{m}_{\lambda_1, \ldots, \lambda_k}\, \chi_{\lambda_1} \otimes \cdots \otimes \chi_{\lambda_k},
\]  
such that $B \in \textnormal{var}^G(A)$, then $\textnormal{Id}^G(A) \subseteq \textnormal{Id}^G(B)$. Consequently, $\Gamma_{n_1, \ldots, n_k}(B)$ can be embedded into $\Gamma_{n_1, \ldots, n_k}(A)$ for all $n=n_1+\cdots+n_k$. Therefore,  $\widetilde{m}_{\lambda_1, \ldots, \lambda_k}\leq m_{\lambda_1, \ldots, \lambda_k}$ for every multipartition $(\lambda_1, \ldots, \lambda_k)\vdash (n_1, \ldots, n_k)$ and $n=n_1+\cdots+n_k$.    
\end{remark}

In the following table, we present the decomposition of $\chi(\Gamma_{n_1, \ldots ,n_k})$ for $n=n_1+\cdots+n_k$, when $n=1$ and $n=2$, together with the corresponding linearized proper h.w.v.s. Here $g,h\in G-\{1\}$ are distinct and recall that $x \circ y$ denotes the Jordan product $xy+yx$.

\begin{table}[h!]
\centering
\begin{tabular}{|c c c c|} 
 \hline
 $\Gamma_{n_1, \ldots, n_k}$ & $\chi(\Gamma_{n_1, \ldots, n_k})$ & linearized proper h.w.v.s & multiplicity  \\ [0.5ex] 
 \hline\hline

 $\Gamma_{(1_g)}$ & $\chi_{((1)_g)}$ & $x_{1,g}$ & 1 \\ \hline

 $\Gamma_{(2_1)}$ & $\chi_{((1,1)_1)}$ & $[x_{1,1},x_{2,1}]$ & 1 \\ \hline 
 $\Gamma_{(2_g)}$ & $\chi_{((2)_g)}$ & $ x_{1,g}\circ x_{2,g}$ & 1 \\  
 & $\chi_{((1,1)_g)}$ & $[x_{1,g},x_{2,g}]$ & 1  \\ \hline
 $\Gamma_{(1_1,1_g)}$ & $\chi_{((1)_1)}\otimes \chi_{((1)_g)}$ & $[x_{1,1},x_{2,g}]$ & 1 \\ \hline
 $\Gamma_{(1_g,1_h)}$ & $2\chi_{((1)_g)}\otimes \chi_{((1)_h)}$ & $[x_{1,g},x_{2,h}],\, x_{1,g}x_{2,h}$ & 2 \\[1ex] 
 \hline
\end{tabular}
\caption{Proper $(n_1, \ldots , n_k)$-cocharacter of $\Gamma_{n_1, \ldots , n_k}$.}
\label{table:1}
\end{table} \vspace{-1cm}

\section{\texorpdfstring{Minimal unitary $G$-graded varieties with quadratic codimension growth}{Minimal unitary G-graded varieties with quadratic codimension growth}}

 Recall that a $G$-graded variety $\mathcal{V}$ is said to be {minimal of polynomial growth $n^t$} if $c_n^G(\mathcal{V}) \approx \alpha n^t$ and for every proper $G$-graded subvariety $\mathcal{U} \subsetneq \mathcal{V}$ we have $c_n^G(\mathcal{U}) \approx \beta n^p$ with $p < t$.

This section is devoted to the construction of important examples of $G$-graded algebras, which are essential for classifying minimal varieties with quadratic codimension growth. We focus on providing a detailed description of their $T_G$-ideals and codimensions.

Given a $G$-graded algebra $A=\bigoplus_{g \in G} A_g$, the {support} of $A$ is defined as $\operatorname{supp}(A) = \{\, g \in G \mid A_g \neq 0 \,\}.$ Henceforth, for each $G$-graded algebra $A$ (the meaning will be clear from the context), the variable $r$ will range over all elements of the set $G - \operatorname{supp}(A)$.

For $m \geq 2$, let $E_1 = \sum\limits_{i = 1}^{m-1} e_{i,i+1} \in UT_{m}$, where $e_{i,j}$ denotes the elementary matrix, with $1$ in the $(i,j)$-entry and zeros elsewhere, and $UT_{m}$ stands for the algebra of $m \times m$ upper triangular matrices. We also denote by $I_{m}$ the $m \times m$ identity matrix. For $g\in G$, consider $C_{m}^g$ the commutative algebra $C_{m}= \{\alpha I_m + \underset{1\leq i<m}{\sum} \alpha_i E_1^i\mid \alpha, \alpha_i \in F\}\subset UT_m$ endowed with the $G$-grading determined by $$I_m\in (C_m)_1\quad \mbox{ and }\quad E_1^i\in (C_m)_{g^i},\quad \mbox{ for }i=1,\ldots , m.$$

Throughout this paper, for an element $g\in G$, we denote its order by $|g|.$

\begin{lemma} \cite{Lamattina} For $g,h\in G$ with $|g|=2$, $|h|>2$ and $u,v\in\{1,h,h^2\}$ we have

\begin{enumerate}
    \item[1)] $\textnormal{Id}^G(C_m^{g})= \langle [x_{1,1},x_{2,1}],[x_{1,1},x_{2,g}],[x_{1,g},x_{2,g}], x_{1,g}\cdots x_{m,g}, x_{1,r} \rangle_{T_G}$. 
    \item[2)] $\textnormal{Id}^{G}(C_{3}^{h})=\langle  [x_{1,u},x_{2,v}] ,x_{1,h}x_{2,h}x_{3,h}, x_{1,h}x_{2,h^2},x_{1,h^2}x_{2,h}, x_{1,h^2}x_{2,h^2} ,x_{1,r} \rangle_{T_{G}}$.
    \item[3)] $c_n^G(C_m^g)=  \sum_{i=0}^{m-1}\binom{n}{i} $ and $C_3^h=1+2n+\binom{n}{2}$.
\end{enumerate}
    
\end{lemma}

Consider $$K_7=F(e_{11}+e_{22}+e_{33})+Fe_{12}+Fe_{13}+Fe_{23}$$ and for distinct elements $g,h\in G-\{1\}$ we define the algebra $K_7^{g,h}$ as the algebra $K_7$ with $G$-grading defined by
 $$e_{11}+e_{22}+e_{33}\in (K_7^{g,h})_1,\quad  e_{12}\in (K_7^{g,h})_g,\quad e_{23}\in (K_7^{g,h})_h\quad \mbox{ and }\quad e_{13}\in (K_7^{g,h})_{gh}.$$

 \begin{lemma} \cite{Cota} For all $g,h,q\in G$ with $g\neq h\neq gh\neq 1$ and $|q|>2$ we have
     \begin{enumerate}
         \item[1)] $\textnormal{Id}^{G}(K_{7}^{g,h})=\langle [x_{1,1},x_{2,1}], [x_{1,1},x_{2,t}], x_{1,t}x_{2,v}, x_{1,r}\mid t, v\in \{g,h, gh\}, (t,v)\neq (g,h)\rangle_{T_G}.$ 

         \item[2)] $\textnormal{Id}^{G}(K_{7}^{q,q^{-1}})=\langle [x_{1,1},x_{2,1}], [x_{1,1},x_{2,t}],x_{1,t}x_{2,v}, x_{1,r}\mid t,v\in \{q, q^{-1}\}, (t,v)\neq (q,q^{-1})\rangle_{T_G}.$

         \item[3)] $c_n^G(K_{7}^{g,h})=1+3n+2\binom{n}{2}$ and $c_n^G(K_{7}^{q,q^{-1}})=1+2n+2\binom{n}{2}$.
     \end{enumerate}
 \end{lemma}

 Consider $\mathcal{G}_{m}=\langle 1,e_1, \ldots , e_m\mid e_ie_i=-e_je_i\rangle$ the subalgebra of the infinite-dimensional Grassmann algebra generated by $1,e_1,\ldots , e_m$ and for $g,h\in G$ with $gh=hg$ let us define $\mathcal{G}_{2}^{g,h}$ as the algebra $\mathcal{G}_2$ with the only $G$-grading such that
$$1\in (\mathcal{G}_{2}^{g,h})_1,\quad  e_{1}\in (\mathcal{G}_{2}^{g,h})_g,\quad  e_{2}\in (\mathcal{G}_{2}^{g,h})_h \quad \mbox{ and }\quad e_{1}e_2\in (\mathcal{G}_{2}^{g,h})_{gh}.$$

\begin{lemma} \cite{Cota, Misso} For $g,u,h\in G-\{1\}$, $w\in \{1,h,h^{2}\}$ and $q\in \{h,h^{2}\}$ we have
\begin{enumerate}
    \item[1)]  $\textnormal{Id}^G(\mathcal{G}_2^{1,1})=\langle [x_{1,1}, x_{2,1}, x_{3,1}], [x_{1,1}, x_{2,1}][x_{3,1}, x_{4,1}],x_{1,r} \rangle_{T_G}.$
    \item[2)] $\textnormal{Id}^G(\mathcal{G}_2^{1,g})=\langle [x_{1,1},x_{2,1}],[x_{1,1},x_{2,g},x_{3,1}],x_{1,g}x_{2,g}, x_{1,r} \rangle_{T_G}.$
    \item[3)] If $|u|=2$ then $\textnormal{Id}^G(\mathcal{G}_2^{u,u})=\langle [x_{1,1}, x_{2,1}], [x_{1,1},x_{2,u}],x_{1,u}x_{2,u}x_{3,u}, x_{1,u}\circ x_{2,u}, x_{1,r}\rangle_{T_G}.$
    \item[4)] If $|h|>2$ then $\textnormal{Id}^G(\mathcal{G}_2^{h,h})=\langle [x_{1,1},x_{2,w}], x_{1,h}\circ x_{2,h}, x_{1,h^2}x_{2,q}, x_{1,q}x_{2,h^2}, x_{1,r} \rangle_{T_G}.$

    \item[5)] $c_n^G(\mathcal{G}_2^{1,1})=1+\binom{n}{2}$, $c_n^G(\mathcal{G}_2^{1,g})=1+n+2\binom{n}{2}$, $c_n^G(\mathcal{G}_2^{u,u})=1+n+\binom{n}{2}$ and $c_n^G(\mathcal{G}_2^{h,h})=1+2n+\binom{n}{2}$.
\end{enumerate}    
\end{lemma}

\begin{lemma} \cite{Cota} Let $g,h,u\in G$, $g\neq h\neq gh\neq 1$, $|u|>2$. For all $s\in \{g,h,gh\}$, $y\in \{u,u^{-1}\}$ we have
    \begin{enumerate}
\item[1)] $\textnormal{Id}^G(\mathcal{G}_{2}^{g,h})= \langle [x_{1,1},x_{2,1}], [x_{1,1},x_{2,s}], x_{1,s}x_{2,s}, x_{1,g}\circ x_{2,h}, x_{1,gh}x_{2,s}, x_{1,s}x_{2,gh},x_{1,r}\rangle_{T_G}$.

       \item[2)] $\textnormal{Id}^G(\mathcal{G}_2^{u,u^{-1}})=\langle [x_{1,1},x_{2,1}],[x_{1,1},x_{2,y}] , x_{1,u}\circ x_{2,u^{-1}}, x_{1,u}x_{2,u}, x_{1,u^{-1}}x_{2,u^{-1}},x_{1,r} \rangle_{T_G}$.

       \item[3)] $c_n^G(\mathcal{G}_2^{g,h})=1+3n+2\binom{n}{2}$ and $c_n^G(\mathcal{G}_2^{u,u^{-1}})=1+2n+2\binom{n}{2}$.
    \end{enumerate}
\end{lemma}

\begin{lemma} Let $g,h,u\in G$, $g\neq h\neq gh\neq 1$, $|u|>2$. For all $s\in \{g,h,gh\}$, $p\in\{1,u,u^{-1}\}$ we have
\begin{enumerate}
    \item[1)] $\textnormal{Id}^G(K_7^{g,h}\oplus \mathcal{G}_2^{g,h})=\langle [x_{1,1},x_{2,1}],  [x_{1,1},x_{2,s}],  x_{1,s}x_{2,s},  x_{1,gh}x_{2,s},  x_{1,s}x_{2,gh}, x_{1,r}\rangle_{T_G}$.
    \item[2)] $\textnormal{Id}^G(K_7^{u,u^{-1}}\oplus \mathcal{G}_2^{u,u^{-1}})=\langle [x_{1,1},x_{2,p}] , x_{1,u}x_{2,u}, x_{1,u^{-1}}x_{2,u^{-1}},x_{1,r}\rangle_{T_G}$.
    \item[3)]The proper nonzero $(n_1, \ldots , n_k)$-cocharacters of $K_7^{g,h}\oplus \mathcal{G}_2^{g,h}$ and $K_7^{u,u^{-1}}\oplus \mathcal{G}_2^{u,u^{-1}}$ are, respectively, 
$$\chi_{((1)_g)} ,\quad  \chi_{((1)_{h})} ,\quad  \chi_{((1)_{gh})}  \quad \mbox{ and }\quad  2\chi_{((1)_g)}\otimes \chi_{((1)_h)} ;$$ 
$$\chi_{((1)_u)} ,\quad \chi_{((1)_{u^{-1}})}  \quad \mbox{ and }\quad   2\chi_{((1)_u)}\otimes \chi_{((1)_{u^{-1}})} .$$ 
    
    \item[4)] $c_n^G(K_7^{g,h}\oplus \mathcal{G}_2^{g,h})=1+3n+4\binom{n}{2}$ and $c_n^G(K_7^{u,u^{-1}}\oplus \mathcal{G}_2^{u,u^{-1}})=1+2n+4\binom{n}{2}$.
\end{enumerate}
\end{lemma}

\begin{proof}
     Since the second case is similar, we focus on the first case. Define $$I=\langle [x_{1,1},x_{2,1}],  [x_{1,1},x_{2,s}],  x_{1,s}x_{2,s},  x_{1,gh}x_{2,s},  x_{1,s}x_{2,gh}, x_{1,r}\rangle_{T_G}$$ and note that $I\subseteq \textnormal{Id}^G(K_7^{g,h}\oplus \mathcal{G}_2^{g,h})$. In order to prove the opposite inclusion, we first notice that $$\Gamma_n^G= \Gamma_n^G \cap \textnormal{Id}^G(K_7^{g,h}\oplus \mathcal{G}_2^{g,h})\subseteq I, \mbox{ for all }n\geq 3.$$ Moreover, it is clear that $$\Gamma_1^G \cap \textnormal{Id}^G(K_7^{g,h}\oplus \mathcal{G}_2^{g,h})=\mbox{span}_F\{x_{1,q}\mid q\in G-\{1,g,h,gh\}\}\subseteq I.$$

     Therefore, we now consider $f$ a multilinear and multihomogeneous proper identity of $K_7^{g,h}\oplus \mathcal{G}_2^{g,h}$ of degree $2$. After reducing $f$ module $I$ we may assume that $f=\alpha x_{1,g} x_{2,h}+\beta [x_{1,g},x_{2,h}]$. Taking the evaluation $x_{1,g}\mapsto e_{1}$ and $x_{1,h}\mapsto e_2$ we obtain $\alpha +2\beta =0$. Now, considering the evaluation $x_{1,g}\mapsto e_{12}$ and $x_{1,h}\mapsto e_{23}$ we obtain $\alpha +\beta =0$. Therefore we must have $\alpha =\beta=0$ and so $I=\textnormal{Id}^G(K_7^{g,h}\oplus \mathcal{G}_2^{g,h})$.

 In order to prove the item $3$, we make use of Proposition \ref{propos}. Notice that the proper highest weight vectors $x_{1,g},x_{1,h}, x_{1,gh}, x_{1,g}x_{2,h}$ and $[x_{1,g},x_{2,h}]$ are not identities of $K_7^{g,h}\oplus \mathcal{G}_2^{g,h}$. Therefore, $$m_{(\lambda)}=1 \quad \mbox{and} \quad 1\leq m_{((1)_g,(1)_h)}\leq 2, \quad \mbox{ for all }\lambda\in \{(1)_g,(1)_h, (1)_{gh}.\}$$ Moreover, by the evaluation above, there is no nonzero linear combination $\alpha x_{1,g} x_{2,h}+\beta [x_{1,g},x_{2,h}]$ resulting in an element of $\textnormal{Id}^G(K_7^{g,h}\oplus \mathcal{G}_2^{g,h})$. Therefore, by Proposition \ref{propos} we have ${m_{((1)_g,(1)_h)}}=2$.

Since $\gamma_{n_1, \ldots , n_k}(A)=\pi_{n_1, \ldots , n_k}(A)(1)$ then item $4$ is a consequence of equations (\ref{propercharac}), (\ref{codprop}) and (\ref{codecodcentral}).
     
\end{proof}

For each element $\alpha \in F^*$, we define an infinite family of algebras 
\[
W_\alpha = F(e_{11} + \cdots + e_{44}) + F(e_{12} + e_{34}) + F(\alpha e_{13} + e_{24}) + F e_{14} \subseteq UT_4.
\]
Let $g,h$ be distinct elements in $ G - \{1\}$ such that $gh=hg$. We define $W_{\alpha}^{g, h}$ as the algebra $W_\alpha$ equipped with a $G$-grading such that:
\[
e_{11} + \cdots + e_{44} \in (W_\alpha^{g, h})_1, \,\,
e_{12} + e_{34} \in (W_\alpha^{g, h})_g,\,\, 
\alpha e_{13} + e_{24} \in (W_\alpha^{g, h})_{h}\,\mbox{ and } \, e_{14} \in (W_\alpha^{g, h})_{gh}.
\]

\begin{lemma} \cite{Cota2} \label{Walpha} For all $u\in \{1,g,g^{-1}\} ,\, v\in \{1,g,h,gh\}$ and $ s\in \{g,h,gh\} $ we have
    \begin{enumerate}
     \item[1)] If $gh=1$ then $\textnormal{Id}^G(W_\alpha ^{g,g^{-1}})= \langle [x_{1,1},x_{2,u}], x_{1,g}x_{2,g}, x_{1,g^{-1}}x_{2,g^{-1}}, \alpha x_{1,g}x_{2,g^{-1}}-x_{2,g^{-1}}x_{1,g}, x_{1,r} \rangle_{T_G}$.

 \item[2)] If $gh\neq 1$ then $\textnormal{Id}^G(W_\alpha ^{g,h})= \langle [x_{1,1},x_{2,v}], x_{1,s}x_{2,s}, x_{1,gh}x_{2,s}, x_{1,s}x_{2,gh}, \alpha x_{1,g}x_{2,h}-x_{2,h}x_{1,g}, x_{1,r} \rangle_{T_G}.$

 \item[3)] $c_n^G(W_\alpha ^{g,g^{-1}})=1+2n+\displaystyle 2\binom{n}{2}$ and $c_n^G(W_\alpha ^{g,h})=1+3n+\displaystyle 2\binom{n}{2}.$
    \end{enumerate}
\end{lemma}

Observe that, when $\alpha =-1$ the map $\varphi: \mathcal{G}_2^{g,h}\rightarrow  W_{-1}^{g,h} $ given by $$\varphi(1_F)=e_{11}+\cdots +e_{44},\quad \varphi(e_1)=e_{12}+e_{34},
\quad \varphi(e_2)=-e_{13}+e_{24}, \quad \varphi(e_1e_2)=e_{14}$$ is an isomorphism of $G$-graded algebras. Since the algebra $\mathcal{G}_2$ plays a significant role in the theory of PI-algebras, we have chosen to treat this algebra separately, rather than incorporating it into the algebras $W_\alpha$.  

We now present the classification of minimal unitary $G$-graded varieties with quadratic growth.

\begin{theorem} \cite{Cota2, Plamen}
    Let $A$ be a unitary minimal $G$-graded algebra such that $c_n^G(A)\leq \alpha n^2$, for some $\alpha$. Then, $A$ is $T_G$-equivalent to one of the following $G$-graded algebras: $C_2^g$, $C_3^g$, $K_7^{r,s}$, $\mathcal{G}_2^{g,h}$ or $W_\alpha ^{u,v}$, where $
    \alpha \in F^*$, $g,h,u,v,r,s\in G$, $r\neq s\neq 1$, $u\neq v\neq 1$, $gh=hg$ and $uv=vu$.
\end{theorem}

Below we present the decomposition of the proper $(n_1, \ldots , n_k)$-cocharacters of the minimal varieties introduced before for $n=n_1+ \cdots + n_k$ when $n=1$ and $n=2$.

\begin{table}[h!]
\centering
\begin{tabular}{|c c c|} 
 \hline
 & $n=1$ & $n=2$  \\[0.5ex] 
 \hline\hline 
 $C_2^g$ & $\chi_{((1)_g)}$ &   \\ 
 $C_3^g,\, |g|=2$ & $\chi_{((1)_g)} $ & $\chi_{((2)_g)} $\\
 $C_3^g,\, |g|>2$ & $\chi_{((1)_g)} $, $\chi_{((1)_{g^2})} $ & $\chi_{((2)_g)} $  \\
 $\mathcal{G}_2^{1,1}$ &  &  $\chi_{((1,1)_1)} $  \\
 $\mathcal{G}_2^{1,g}$ & $\chi_{((1)_g)} $ & $\chi_{((1)_1)}\otimes \chi_{((1)_g)} $ \\
 $\mathcal{G}_2^{g,g},\, |g|=2$ & $\chi_{((1)_g)} $ & $\chi_{((1,1)_g)} $  \\
 $\mathcal{G}_2^{g,g},\, |g|>2$ & $\chi_{((1)_g)} $, $\chi_{((1)_{g^2})} $ & $\chi_{((1,1)_g)} $  \\
 $K_7^{g,h},\, gh\neq 1$ & $\chi_{((1)_g)} $, $\chi_{((1)_{h})} $, $\chi_{((1)_{gh})} $ & $\chi_{((1)_g)}\otimes \chi_{((1)_h)} $  \\  $K_7^{g,h},\, gh= 1$ & $\chi_{((1)_g)} $, $\chi_{((1)_{h})} $ & $\chi_{((1)_g)}\otimes \chi_{((1)_h)} $  \\
 $\mathcal{G}_2^{g,h},\, gh\neq 1$ & $\chi_{((1)_g)} $, $\chi_{((1)_{h})} $, $\chi_{((1)_{gh})} $ & $\chi_{((1)_g)}\otimes \chi_{((1)_h)} $  \\
 $\mathcal{G}_2^{g,h},\, gh= 1$ & $\chi_{((1)_g)} $, $\chi_{((1)_{h})} $,  & $\chi_{((1)_g)}\otimes \chi_{((1)_h)} $  \\
 $W_{\alpha}^{g,h},\, gh\neq 1$ & $\chi_{((1)_g)} $, $\chi_{((1)_{h})} $, $\chi_{((1)_{gh})} $ & $\chi_{((1)_g)}\otimes \chi_{((1)_h)} $  \\
 $W_{\alpha}^{g,h},\, gh= 1$ & $\chi_{((1)_g)} $, $\chi_{((1)_{h})} $ & $\chi_{((1)_g)}\otimes \chi_{((1)_h)} $ \\[1ex] 
 \hline
\end{tabular}
\caption{Proper $(n_1, \ldots , n_k)$-cocharacters of minimal varieties.}
\label{table:2}
\end{table}

\section{Characterizing varieties with nonzero multiplicities}

In this work we are interested in unitary $G$-graded varieties $\textnormal{var}^G(A)$ with quadratic codimension growth. To this end, we establish a direct connection between the nonzero multiplicities in the proper $(n_1, \ldots, n_k)$-cocharacters and the minimal varieties.

From now on, we assume that $A$ is a unitary $G$-graded algebra with quadratic codimension growth. By Theorem \ref{codproper}, this means that $\gamma_i^G(A)=0$ for all $i\geq 3$, and so $\Gamma_i^G\subseteq \textnormal{Id}^G(A)$ for all $i\geq 3$. Therefore, by the previous discussion we only need to investigate the subspaces $\Gamma_1^G$ and $\Gamma_2^G$, which are completely determined by their multihomogeneous components $\Gamma_{n_1, \ldots , n_k}$, $n=n_1+\cdots +n_k$ when $n=1$ and $n=2$. 

Since $\Gamma_{n_1, \ldots, n_k}(A)$ can be seen as an $S_{n_1, \ldots, n_k}$-submodule of $\Gamma_{n_1, \ldots, n_k}$, the multiplicities in the decomposition \eqref{propercharac} of $\pi_{n_1, \ldots , n_k}(A)$ are bounded by the multiplicities of the corresponding irreducible character in the decomposition of $\chi(\Gamma_{n_1, \ldots, n_k})$. Therefore, by Table \ref{table:1}, for a unitary $G$-graded algebra with quadratic codimension growth and proper $(n_1, \ldots , n_k)$-cocharacter as in \eqref{propercharac}, we have 
\begin{equation} \label{mult}
    0\leq m_{\lambda}\leq 1, \quad m_{\mu}=0 \quad \mbox{and}\quad 0\leq m_{((1)_g,(1)_h)}\leq 2,
\end{equation}
for all $\lambda\in S=\{((1)_g),((1,1)_1), ((2)_g), ((1,1)_g), ((1)_1, (1)_g)\}$ with $g,h\in G-\{1\}$, $g\neq h$, and $\mu \notin S$. Moreover, we necessarily have $m_{\lambda}\neq 0$ for some multipartition $\lambda$ of $2$.

The following remark, though simple, is fundamental and deserves special attention.

\begin{remark} \label{remark}
    Let $A$ and $B$ be unitary $G$-graded algebras with quadratic growth and assume that $B\in \textnormal{var}^G(A)$. If $B$ and $A$ have the same multiplicities in the decomposition of all proper $(n_1, \ldots , n_k)$-cocharacters, then $\textnormal{var}(A)=\textnormal{var}(B)$. 
\end{remark}

\begin{proof}
    In fact, if $B\in \textnormal{var}^G(A)$ then $\textnormal{Id}^G(A)\subseteq \textnormal{Id}^G(B)$. Since $A$ and $B$ are unitary $G$-graded algebras with quadratic growth, to prove the opposite inclusion, it is sufficient to show that $\Gamma_n^G\cap \textnormal{Id}^G(B)\subseteq \Gamma_n^G\cap \textnormal{Id}^G(A)$ for all $n=1,2$. Since the multiplicities are the same, we must have $\Gamma_{(n_1, \ldots , n_k)}\cap \textnormal{Id}^G(B)= \Gamma_{(n_1, \ldots , n_k)}\cap \textnormal{Id}^G(A)$ for all $n=n_1+ \cdots +n_k$, when $n=1$ and $n=2$. As the polynomial identities of $B$ follow from its proper multilinear and multihomogeneous identities, the result follows.
\end{proof}

Motivated by Theorem~\ref{35}, in the subsequent results we assume that $A$ is a finite-dimensional unitary $G$-graded algebra with quadratic codimension growth of the type $F+J(A)$, where $J(A)$ denotes the Jacobson radical of $A$. Recall that $J(A)$ is a $G$-graded ideal and, therefore, decomposes into homogeneous components as
\[
J(A)=\bigoplus_{g\in G} J(A)_g.
\]

The next step is to analyze all possible nonzero values of the multiplicities $m_{\lambda_1,\ldots , \lambda_k}$ in (\ref{mult}). More specifically, we characterize the possible values of these multiplicities in terms of an algebra that generates a minimal variety within the given variety.

\begin{lemma} \label{c2g}
     $m_{((1)_g)}\neq 0$ if and only if $C_2^{g}\in \textnormal{var}^G(A)$.
\end{lemma}

\begin{proof}
    Assume that  $m_{((1)_g)}\neq 0$. According to Proposition \ref{propos} we have $x_{1,g}\neq 0$. Therefore, there exists a nonzero element $a\in J(A)_g$. Let $R$ be the $G$-graded subalgebra of $A$ generated by $1_F$ and $a$ and consider $I$ the $T_G$-ideal of $R$ generated by $a^2$. Hence, $R/I$ is a $G$-graded algebra that is isomorphic to $C_2^g$ through the isomorphism of $G$-graded algebras given by $\overline{1_F}\mapsto e_{11}+e_{22}$ and $\overline{a}\mapsto e_{12}$. Therefore, $C_2^g\cong R/I\in \textnormal{var}^G(R)\subseteq \textnormal{var}^G(A).$

    Clearly, by Remark \ref{multibounded} and Table \ref{table:2}, if $C_2^{g}\in \textnormal{var}^G(A)$ then  $m_{((1)_g)}\neq 0$.
\end{proof}

\begin{lemma} \label{g211}
 $m_{((1,1)_1)}\neq 0$ if and only if $\mathcal{G}_2^{1,1}\in \textnormal{var}^G(A)$.
\end{lemma}

\begin{proof}
    In fact, if $m_{((1,1)_1)}\neq 0$ then by Proposition \ref{propos} we have $[x_{1,1},x_{1,1}]\notin \textnormal{Id}^G(A)$. Since $B=F+J(A)_1$ is a $G$-graded subalgebra of $A$, we have $[x_{1,1},x_{1,1}]\notin \textnormal{Id}^G(B)$ and so, by \cite[Lemma 9]{GLa}, we get 
    $$\mathcal{G}_2^{1,1}\in \textnormal{var}^G(B)\subseteq \textnormal{var}^G(A).$$

    On the other hand, it follows from  Remark \ref{multibounded} and Table \ref{table:2} that if $\mathcal{G}_2^{1,1}\in \textnormal{var}^G(A)$ then we must have $m_{((1,1)_1)}\neq 0$ and so the proof follows. 
\end{proof}

\begin{lemma}
 $m_{((1)_1,(1)_g)}\neq 0$  if and only if $\mathcal{G}_2^{1,g}\in \textnormal{var}^G(A)$.
\end{lemma}

\begin{proof}
If $\mathcal{G}_2^{1,g}\in \textnormal{var}^G(A)$, then by Remark~\ref{multibounded} and Table~\ref{table:2}, it is clear that $m_{((1)_1,(1)_g)}\neq 0$. 

Conversely, assume that $m_{((1)_1,(1)_g)}\neq 0$. Since both $A$ and $\mathcal{G}_2^{1,g}$ have quadratic codimension growth, it follows that  
\[
\Gamma_n^G = \Gamma_n^G \cap \textnormal{Id}^G(A) \subseteq \textnormal{Id}^G(\mathcal{G}_2^{1,g}), 
\quad \text{for all } n\geq 3.
\]  Moreover, as long as $\textnormal{supp}(\mathcal{G}_2^{1,g})=\{1,g\}$ and $x_{1,g}\notin \textnormal{Id}^G(A)$, we obtain  
\[
\Gamma_1^G \cap \textnormal{Id}^G(A)
\subseteq \textnormal{span}_F \{x_{1,h}\mid h\in G-\{1,g\}\}
\subseteq \textnormal{Id}^G(\mathcal{G}_2^{1,g}).
\]  

It remains to analyze the proper multilinear and multihomogeneous $G$-graded identities of degree $2$ of $A$.  
Assume that $f\in \Gamma_{n_1, \ldots ,n_k}\cap \textnormal{Id}^G(A)$, with $n_1+\cdots +n_k=2$.  
Since  $[x_{1,1},x_{2,1}]$, $x_{1,g}x_{2,g}$ and $x_{1,r}$, are identities of $\mathcal{G}_2^{1,g}$, for all $r\in G-\{1,g\}$, it follows that either $f\in \textnormal{Id}^G(\mathcal{G}_2^{1,g})$ or $f\in \Gamma_{(1_1,1_g)}$.  

Observe that, by Proposition~\ref{propos}, we have $[x_{1,1},x_{2,g}]\notin \textnormal{Id}^G(A)$, and so the second case cannot occur. Therefore,  $\Gamma_2^G \cap \textnormal{Id}^G(A) \subseteq \textnormal{Id}^G(\mathcal{G}_2^{1,g})$. Since  $\textnormal{Id}^G(A)$ is generated by its multilinear proper identities we have $\textnormal{Id}^G(A)\subseteq \textnormal{Id}^G(\mathcal{G}_2^{1,g})$ and so $\mathcal{G}_2^{1,g}\in \textnormal{var}^G(A).$ 
\end{proof}



\begin{lemma}
     $m_{((2)_g)}\neq 0$ if and only if $C_3^g\in \textnormal{var}^G(A)$. 
\end{lemma}

\begin{proof}
    Assume $m_{((2)_g)}\neq 0$ then $x_{1,g}\circ x_{2,g}\notin \textnormal{Id}^G(A)$. Therefore, there exist $a\in J(A)_g$ such that $a^2\neq 0$. Consider $R$ the $G$-graded subalgebra of $A$ generated by $1$ and $a$. The quotient algebra $R/I$, where $I$ is the $T_G$-ideal generated by $a^3$, is isomorphic to $C_3^g$ as $G$-graded algebras through the isomorphism $\overline{1}\mapsto e_{11}+e_{22} +e_{33},$ $ \overline{a}\mapsto e_{12}+e_{23}$ and $\overline{a^2}\mapsto e_{13}.$ Therefore, $C_3^g\cong R/I\in \textnormal{var}^G(R)\subseteq \textnormal{var}^G(A).$

    If $C_3^g\in \textnormal{var}^G(A)$, by Remark \ref{multibounded} and Table \ref{table:2}, it is clear that $m_{((2)_g)}\neq 0$. 
\end{proof}

\begin{lemma}
 $m_{((1,1)_g)}\neq 0$ if and only if $\mathcal{G}_2^{g,g}\in \textnormal{var}^G(A)$. 
\end{lemma}

\begin{proof}
 We begin by observing that if $m_{((1,1)_g)} \neq 0$ and $m_{((2)_g)} = 0$, then, by Proposition \ref{propos} and Table \ref{table:1}, there exist $a, b \in J(A)_g$ such that $[a,b] \neq 0$ and $a^2 = b^2 = ab + ba = 0.$

Let $R$ be the $G$-graded algebra generated by $1_F$, $a$, and $b$. Note that $R \cong \mathcal{G}_2^{g,g}$ via the isomorphism of $G$-graded algebras defined by
\[
1_F \mapsto 1, \quad a \mapsto e_1, \quad b \mapsto e_2, \quad ab \mapsto e_1 e_2.
\]

Therefore, we may now assume that $m_{((1,1)_g)} \neq 0$ and $m_{((2)_g)} \neq 0$. In other words, both $[x_{1,g},x_{2,g}]$ and $x_{1,g} \circ x_{2,g}$ are not identities of $A$. We claim that in this case $C_3^g \oplus \mathcal{G}_2^{g,g} \in \textnormal{var}^G(A).$

If $|g| = 2$, then $B = F + J(A)_1 + J(A)_g$ is a $G$-graded subalgebra of $A$ with an induced $\mathbb{Z}_2$-grading. By \cite[Lemma 5.11]{Cota2}, if $[x_{1,g},x_{2,g}]$ and $x_{1,g} \circ x_{2,g}$ are not identities of $B$, then $C_3^g \oplus \mathcal{G}_2^{g,g} \in \textnormal{var}^G(B) \subseteq \textnormal{var}^G(A)$ and the result follows.

Now assume $|g| > 2$. Since both $A$ and $C_3^g \oplus \mathcal{G}_2^{g,g}$ have quadratic codimension growth, it follows that for all $n \geq 3$,
\[
\Gamma_n^G = \Gamma_n^G \cap \textnormal{Id}^G(A) = \Gamma_n^G \cap \textnormal{Id}^G(C_3^g \oplus \mathcal{G}_2^{g,g}).
\]
Furthermore, since $x_{1,g}$ and $x_{1,g^2}$ are not identities of $A$ and
$\textnormal{supp}(C_3^g \oplus \mathcal{G}_2^{g,g}) = \{1, g, g^2\}$, we have
\[
\Gamma_{1}^G \cap \textnormal{Id}^G(A) \subseteq \Gamma_{1}^G \cap \textnormal{Id}^G(C_3^g \oplus \mathcal{G}_2^{g,g}).
\]

Thus, it remains to analyze the multilinear and multihomogeneous $G$-graded identities of degree $2$ of $A$. Let $f \in \Gamma_2^G \cap \textnormal{Id}^G(A)$ be a multihomogeneous identity. Since the polynomials
\[
[x_{1,1},x_{2,1}], \quad [x_{1,1},x_{2,g}], \quad x_{1,g^2}x_{2,g}, \quad x_{1,g}x_{2,g^2}, \quad x_{1,g^2}x_{2,g^2}, \quad x_{1,h}\] are identities of   $C_3^g \oplus \mathcal{G}_2^{g,g}$,
for all $h \in G - \{1, g, g^2\}$, we may assume that $f \in \Gamma_{(2_g)}$; otherwise, it is immediate that $f \in \textnormal{Id}^G(C_3^g \oplus \mathcal{G}_2^{g,g})$. Therefore, we may write
\[
f = \alpha \, x_{1,g}x_{2,g} + \beta \, [x_{1,g},x_{2,g}].
\]
Consider the $G$-graded endomorphism sending $x_{1,g}, x_{2,g} \mapsto x_{1,g}$. Then $\alpha \, x_{1,g}^2 \equiv 0 \quad \text{in } A.$ Since $x_{1,g} \circ x_{2,g} \notin \textnormal{Id}^G(A)$, we conclude that $\alpha = 0$. Similarly, as $[x_{1,g},x_{2,g}] \notin \textnormal{Id}^G(A)$, we must have $\beta = 0$. Hence, it follows that
$$
\Gamma_n^G \cap \textnormal{Id}^G(A) \subseteq \Gamma_n^G \cap \textnormal{Id}^G(C_3^g \oplus \mathcal{G}_2^{g,g}),$$ for all $n$. Since $A$ is unitary, we conclude that $C_3^g \oplus \mathcal{G}_2^{g,g} \in \textnormal{var}^G(A)$ in this case.

Finally, note that
\[
\textnormal{Id}^G(A) \subseteq \textnormal{Id}^G(C_3^g \oplus \mathcal{G}_2^{g,g}) = \textnormal{Id}^G(C_3^g) \cap \textnormal{Id}^G(\mathcal{G}_2^{g,g}) \subseteq \textnormal{Id}^G(\mathcal{G}_2^{g,g}),
\]
and therefore $\mathcal{G}_2^{g,g} \in \textnormal{var}^G(A).$
\end{proof}

As a consequence of the previous results we have the following. 

\begin{corollary}
    $m_{((2)_g)}, m_{((1,1)_g)}\neq 0$ if and only if $\mathcal{G}_2^{g,g}\oplus C_3^g\in \textnormal{var}^G(A).$
\end{corollary}

Now we investigate the cases where the multiplicities $m_{((1)_g(1)_h)}$ is nonzero. Therefore, in the following we assume that $A=F+J(A)$ is a unitary $G$-graded algebra with quadratic codimension growth such that $x_{1,g}x_{2,h}\notin \textnormal{Id}^G(A)$, for distinct elements $g,h\in G-\{1\}$. 

\begin{lemma} \label{xgxh}
  If $ \alpha x_{1,g}x_{2,h}- x_{2,h}x_{1,g}\equiv 0$ on $A$, for some $\alpha \in F$, then either $\alpha =0$ and $K_7^{g,h}\in \textnormal{var}^G(A)$  or $gh=hg$ and $W_{\alpha}^{g,h}\in \textnormal{var}^G(A).$  
\end{lemma}

\begin{proof}
    Let $a\in J(A)_g$ and $b\in J(A)_h$ such that $ab\neq 0$ and $ba= \alpha ab$. Consider $R$ the $G$-graded subalgebra of $A$ generated by $1_F$, $a$ and $b$ and let $I$ be the $T_G$-ideal of $R$ generated by $a^2$ and $b^2$. 

Since $\Gamma_n^G\subseteq \textnormal{Id}^G(A)$, for all $n\geq 3$, $I$ is linearly generated by $a^2$ and $b^2$. Since $g\neq h\neq 1$ then $ab\notin I$. Therefore, $a,b,ab\notin I$ and thus $R/I$ is a $G$-graded algebra satisfying $\overline{a}^2=\overline{b}^2=\alpha \overline{a}\overline{b}- \overline{b}\overline{a}=0$.

If $\alpha=0$ then $R/I\cong K_7^{g,h}$ through the map $$\overline{1_F}\mapsto e_{11}+e_{22}+e_{33},\quad  \overline{a}\mapsto e_{12},\quad \overline{b}\mapsto e_{23},\quad   \overline{ab}\mapsto e_{13}.  $$

Otherwise, we have $gh=hg$ and $R/I\cong W_{\alpha}^{g,h}$ through the map $$\overline{1_F}\mapsto e_{11}+\cdots +e_{44},\quad \overline{a}\mapsto e_{12}+e_{34},\quad \overline{b}\mapsto \alpha e_{13}+e_{24},\quad  \overline{ab}\mapsto e_{14}.  $$

Since $R/I \in \textnormal{var}^G(R)\subseteq \textnormal{var}^G(A)$ the result follows.

\end{proof}

\begin{lemma}
If $m_{((1)_g(1)_h)}=1$ then either $K_7^{h,g}$ or $K_7^{g,h}$ or $\mathcal{G}_2^{g,h}$ or $W_\alpha ^{g,h}\in \textnormal{var}^G(A)$, $\alpha \in F- \{0,-1\}$. 
\end{lemma}

\begin{proof}

Assume that $m_{((1)_g(1)_h)} = 1$. Then, by Proposition~\ref{propos}, we must analyze three cases:

\smallskip
\noindent
\textbf{Case 1.} If $x_{1,g}x_{2,h} \equiv 0$ and $[x_{1,g},x_{2,h}] \not\equiv 0$ on $A$, then, by the previous lemma, we obtain $K_7^{h,g} \in \textnormal{var}^G(A)$.

\smallskip
\noindent
\textbf{Case 2.} If $x_{1,g}x_{2,h} \not\equiv 0$ and $[x_{1,g},x_{2,h}] \equiv 0$ on $A$, then Lemma~\ref{xgxh} yields $W_1^{g,h} \in \textnormal{var}^G(A)$.

\smallskip
\noindent
\textbf{Case 3.} Suppose there exists a linear combination $f = x_{1,g}x_{2,h} + \beta \,[x_{1,g},x_{2,h}] \in \textnormal{Id}^G(A)$,
with $\beta \neq 0$. In this case,
\[
\frac{1+\beta}{\beta} \, x_{1,g}x_{2,h} - x_{2,h}x_{1,g} \equiv 0
\] on $A$. Note that if $\beta=-1$ then by Lemma~\ref{xgxh} we have $K_7^{g,h}\in \textnormal{var}^G(A)$. Otherwise, we must have $gh=hg$ and, by Lemma~\ref{xgxh}, we have $W_{\alpha}^{g,h} \in \textnormal{var}^G(A)$, where $\alpha = \frac{1+\beta}{\beta}$. In particular, if $\beta = -\frac{1}{2}$, then $\alpha = -1$ and $\mathcal{G}_2^{g,h} \in \textnormal{var}^G(A)$.

\end{proof}

\begin{lemma} \label{ultimo}
 $m_{((1)_g(1)_h)}=2$ if and only if $K_7^{g,h}\oplus \mathcal{G}_2^{g,h} \in \textnormal{var}^G(A)$. 
\end{lemma}

\begin{proof}

First, note that since $A$ and $K_7^{g,h}\oplus \mathcal{G}_2^{g,h}$ have quadratic growth, then, for all $n\geq 3$, we have $$\Gamma_n^G=\Gamma_n^G\cap \textnormal{Id}^G(A)=\Gamma_n^G\cap \textnormal{Id}^G(K_7^{g,h}\oplus \mathcal{G}_2^{g,h}).$$ Moreover, since $x_{1,s}\notin \textnormal{Id}^G(A)$, for all $s\in \{1,g,h,gh\}$ and $\mbox{supp}(K_7^{g,h}\oplus \mathcal{G}_2^{g,h})=\{1,g,h,gh\}$ then we also have $$\Gamma_1^G\cap \textnormal{Id}^G(A)\subseteq \textnormal{Id}^G(K_7^{g,h}\oplus \mathcal{G}_2^{g,h}).$$ Therefore, in order to prove that $K_7^{g,h}\oplus \mathcal{G}_2^{g,h} \in \textnormal{var}^G(A)$ we just need to analyze the multilinear and multihomogeneous identities of degree $2$ of $A$.  

Assume that $gh\neq 1$ and consider $f\in \Gamma_2^G\cap \textnormal{Id}^G(A)$ a multihomogeneous identity. Since the polynomials
$$[x_{1,1},x_{2,1}],\quad  [x_{1,1},x_{2,s}],\quad  x_{1,s}x_{2,s}, \quad x_{1,gh}x_{2,s}, \quad x_{1,s}x_{2,gh}\quad \mbox{ and }\quad x_{1,r}$$ are identities of $K_7^{g,h}\oplus \mathcal{G}_2^{g,h}$, for all $s\in \{g,h,gh\}$ and $r\in G-\{1,g,h,gh\}$, then $f\equiv 0$ on $K_7^{g,h}\oplus \mathcal{G}_2^{g,h}$ or  $f=\alpha x_{1,g} x_{2,h}+\beta [x_{1,g},x_{2,h}]$. Since $m_{((1)_g(1)_h)}=2$, by Proposition \ref{propos}, there are no $\alpha , \beta \in F$, $(\alpha, \beta)\neq (0,0)$, such that $\alpha x_{1,g}x_{2,h}+\beta [x_{1,g},x_{2,h}]\equiv 0 $ on $A$ and so we must have $f\in \textnormal{Id}^G(K_7^{g,h}\oplus \mathcal{G}_2^{g,h}).$

If $gh=1$ we note that the polynomials
$$[x_{1,1},x_{2,1}],\quad  [x_{1,1},x_{2,s}],\quad  x_{1,s}x_{2,s} \quad  \mbox{ and }\quad x_{1,r}$$ are identities of $K_7^{g,h}\oplus \mathcal{G}_2^{g,h}$, for all $s\in \{g,h\}$ and $r\in G-\{1,g,h\}$, then the proof follows similarly to the previous case.

Therefore, we have proved that $$\Gamma_n^G\cap \textnormal{Id}^G(A)=\Gamma_n^G\cap \textnormal{Id}^G(K_7^{g,h}\oplus \mathcal{G}_2^{g,h}),$$ for all $n$. Since $A$ is unitary, then we have $\textnormal{Id}^G(A)\subseteq \textnormal{Id}^G(K_7^{g,h}\oplus \mathcal{G}_2^{g,h})$ and so $K_7^{g,h}\oplus \mathcal{G}_2^{g,h}\in \textnormal{var}^G(A)$.

Conversely, if $K_7^{g,h}\oplus \mathcal{G}_2^{g,h} \in \textnormal{var}^G(A)$ then $2\leq m_{((1)_g(1)_h)}\leq \dim \Gamma^G_{((1)_g(1)_h)}=2.$
    
\end{proof}

Given an algebra \(A\), consider the direct product 
\(\widetilde{A} = A \times F\) equipped with the multiplication
\[
(a_1,\alpha_1)\,(a_2,\alpha_2)
= \big(a_1 a_2 + \alpha_1 a_2 + \alpha_2 a_1,\; \alpha_1\alpha_2 \big).
\]
This construction yields the unitary algebra obtained from \(A\) 
by adjoining an identity element. If \(A\) is a \(G\)-graded algebra,  then \(\widetilde{A}\) inherits a natural \(G\)-grading by defining
\[
\widetilde{A}_1 = (A_1, F)
\qquad\text{and}\qquad
\widetilde{A}_g = (A_g, \{0\}),
\quad\text{for all } g \neq 1.
\]

\begin{lemma} \label{btildeinvar}
    Let $A$ be a unitary $G$-graded algebra. If $B\in \textnormal{var}^
    G(A)$ then $\widetilde{B}\in \textnormal{var}^G(A)$.
\end{lemma}

\begin{proof}
 Let $f(x_{1,g_{i_1}}, \ldots, x_{n,g_{i_n}}) \in \textnormal{Id}^{G}(A)$ be a multilinear \(G\)-graded polynomial of degree \(n\). Since both 
\(A\) and \(\widetilde{A}\) are unitary \(G\)-graded algebras, we may assume, 
without loss of generality, that \(f\) is a proper \(G\)-graded polynomial. 
Since \(f\) is proper and multilinear, any evaluation on elements 
\((a_{q,g_{i_q}},\alpha_q)\in \widetilde{A}_{g_{i_q}}\), \(1\le q\le n\), 
satisfies
\[
f((a_{1,g_{i_1}},\alpha_1),\ldots,(a_{n,g_{i_n}},\alpha_n))
=
f((a_{1,g_{i_1}},0),\ldots,(a_{n,g_{i_n}},0)).
\]
Therefore the evaluation reduces to values in \(A\), and the conclusion follows 
from the fact that \(f\equiv 0\) on \(A\).

\end{proof}

\begin{lemma} \label{unitariornilpotentorcummu}
    Let $A$ be a unitary finite-dimensional $G$-graded algebra of polynomial codimension growth without unity. Then either $
    \widetilde{A}$ has exponential growth or $A\sim_{T_{G}}N$ or $A\sim_{T_{G}} C\oplus N$, where $N$ is a nilpotent $G$-graded algebra and $C$ is a commutative non-nilpotent algebra with trivial grading.  
\end{lemma}

\begin{proof}

Since \(A\) is a finite-dimensional \(G\)-graded algebra with polynomial 
codimension growth, by \cite[Theorem~9]{Valenti} and 
\cite[Theorem~2.2]{AljGiam} we may assume that
\[
A = A_1 \oplus \cdots \oplus A_l \,+ J,
\]
where \(J\) denotes the Jacobson radical of \(A\), each 
\(A_i \cong F\) endowed with the trivial grading, and 
\(A_t J A_m = 0\) for all \(t \neq m\). Moreover, if the decomposition 
\(A = A_1 \oplus \cdots \oplus A_l \oplus J\) 
is a direct sum of ideals, then \(A \sim_{T_G} C \oplus N\) or 
\(A \sim_{T_G} N\).

Otherwise, there exists \(1 \leq i \leq l\) such that \(A_i J \neq 0\). 
Considering the extension 
\(\widetilde{A} = (A_1 \oplus \cdots \oplus A_l \,+\, J) \times F\), 
we obtain
\[
\widetilde{A}
=
\bar{A}_1 \oplus \cdots \oplus \bar{A}_l 
\oplus \bar{F} \,+\, \bar{J},
\] where $\bar{A}_i = \{(a_i,0) \mid a_i \in A_i\}, $ $
\bar{F} = \{(0,\alpha) \mid \alpha \in F\} \cong F$ and $
\bar{J} = \{(j,0) \mid j \in J\}.$

For \(t \neq m\) we have \(\bar{A}_t\,\bar{J}\,\bar{A}_m = \{0\}\). 
However, since
\[
\bar{A}_i\,\bar{J}\,\bar{F}
= \bar{A}_i\,\bar{J} \neq \{0\},
\]
\cite[Theorem~2.2]{AljGiam} implies that \(\widetilde{A}\) has 
exponential codimension growth.    
\end{proof}

\begin{corollary} 
\label{unitarunilpotentcomm}
    Let $A$ be a unitary $G$-graded algebra of polynomial codimension growth. If $B\in \textnormal{var}(A)$ is a finite-dimensional algebra then either $B$ is unitary or $B\sim_{T_G}N$ or $B\sim_{T_G}C\oplus N$, where $C$ is a commutative non-nilpotent $G$-graded algebra with trivial grading and $N$ is a nilpotent $G$-graded algebra.
\end{corollary}

\begin{proof}
Let \(B\) be a nonunitary \(G\)-graded algebra such that 
\(B \in \textnormal{var}^G(A)\). Since \(A\) has polynomial codimension growth,  it follows that \(B\) also has polynomial growth. Moreover, by  Lemma~\ref{btildeinvar} we have \(\widetilde{B} \in \textnormal{var}^G(A)\), and  therefore \(\widetilde{B}\) has polynomial growth as well. The proof 
now follows from Lemma~\ref{unitariornilpotentorcummu}.
\end{proof}

\black

Finally, we can prove the main result of this paper.

\begin{theorem} \label{quadraticgrowth}
    Let $\textnormal{var}^G(A)$ be a unitary $G$-graded variety over a field $F$ of characteristic zero with quadratic codimension growth. Then, $A$ is $T_G$-equivalent to a finite direct sum of $G$-graded algebras in the set 
$$\{F, \, C_2^q,C_3^q, K_7^{p,q} ,\mathcal{G}_2^{g,h}, W_\alpha ^{u,v} \mid \alpha \in F^*,gh=hg, uv=vu, p\neq q\neq 1, u\neq v\neq 1  \},$$
 where at least one algebra in the set $\{C_3^s, K_7^{p,q}, \mathcal{G}_2^{g,h}, W_\alpha^{u,v} \mid s \neq 1\}$ appears as a direct summand.
\end{theorem}

\begin{proof}
Since $A$ is a unitary algebra with polynomial growth, Theorem~\ref{35} implies that
\[
A \sim_{T_G} A_1 \oplus \cdots \oplus A_m,
\] where each $A_i$ is a finite-dimensional $G$-graded algebra, either nilpotent or of the form $F + J(A_i)$. Since $A$ has quadratic codimension growth, the direct sum $A_1 \oplus \cdots \oplus A_m$ also has quadratic codimension growth. Hence, by Lemma~\ref{unitarunilpotentcomm}, the algebra $A_1 \oplus \cdots \oplus A_m$ is unitary. In particular, $A_t = F + J(A_t)$ for every $t = 1, \ldots, m$, and there exists an index $i$ such that $A_i$ has quadratic growth of the sequence of $G$-graded codimensions.  In this case, by Corollary \ref{unitarunilpotentcomm}, we may assume that $A_i$ is unitary and so we have $\Gamma_n^G \subseteq \mathrm{Id}^G(A_i)$, for all $n \ge 3$.

We now analyze the possible values of the multiplicities appearing in the proper $(n_1, \ldots, n_k)$-cocharacters of $A_i$, where $n = n_1 + \cdots + n_k$ and $n \in \{1,2\}$. According to Table~\ref{table:1}, these multiplicities satisfy
\[
m_{((1,1)_1)} = 0, \qquad 0 \le m_{((1)_1,(1)_g)} \le 1, \qquad 0 \le m_{((2)_g)} \le 1, \qquad 0 \le m_{((1,1)_g)} \le 1, \qquad 0 \le m_{((1)_g,(1)_h)} \le 2,
\]
for all $g,h \in G - \{1\}$ with $g \neq h$.

For $n = 1$ and $g \in G - \{1\}$, Lemma~\ref{c2g} ensures that either $m_{((1)_g)} = 0$ or $m_{((1)_g)} = 1$ and $C_2^g \in \mathrm{var}^G(A_i)$.

When $n = 2$, Lemmas~\ref{g211}--\ref{ultimo} determine all possible cases. If $m_{((1,1)_1)} = 1$, then $\mathcal{G}_2^{1,1} \in \mathrm{var}^G(A_i)$; if $m_{((1)_1,(1)_g)} = 1$, then $\mathcal{G}_2^{1,g} \in \mathrm{var}^G(A_i)$.
Similarly, $m_{((2)_g)} = 1$ implies $C_3^g \in \mathrm{var}^G(A_i)$, while $m_{((1,1)_g)} = 1$ implies $\mathcal{G}_2^{g,g} \in \mathrm{var}^G(A_i)$.
Finally, for distinct $g,h \in G - \{1\}$, if $m_{((1)_g,(1)_h)} = 1$, then $R\in \textnormal{var}^G(A_i)$ for some $R \in \{K_7^{h,g}, K_7^{g,h}, \mathcal{G}_2^{g,h}, W_\alpha^{g,h} \mid \alpha \in F - \{0,-1\}\}$.
In the case where $m_{((1)_g,(1)_h)} = 2$, we have $K_7^{g,h} \oplus \mathcal{G}_2^{g,h} \in \mathrm{var}^G(A_i)$. 

Moreover, since $A_i$ has quadratic growth, at least one of the multiplicities $m_{\lambda}$ must be nonzero, where
\[
\lambda \in \{(1,1)_1,\ ((1)_1,(1)_g),\ ((2)_g),\ ((1,1)_g),\ ((1)_g,(1)_h) \mid g,h \in G - \{1\},\ g \neq h\}.
\]

Note that for each nonzero multiplicity, there is a corresponding $G$-graded algebra in the variety that contributes to this value. Let $B$ denote the $G$-graded algebra obtained as the direct sum of these algebras, selected according to the nonzero multiplicities of $A_i$. 
By Lemmas~\ref{c2g}--\ref{ultimo}, it follows that $B \in \mathrm{var}^G(A_i)$.

Furthermore, by Remark~\ref{multibounded} and Table~\ref{table:2}, the algebras $A_i$ and $B$ share the same multiplicities in the decomposition of all proper $(n_1, \ldots, n_k)$-cocharacters. 
Hence, by Remark~\ref{remark}, we conclude that
\[
\mathrm{var}^G(A_i) = \mathrm{var}^G(B).
\]

We now observe that, if $A_t$ has at most linear codimension growth, for some $1\leq t\leq m$, then by \cite[Theorem~5.5]{Plamen}, $A_t$ is $T_G$-equivalent to a finite direct sum of $G$-graded algebras belonging to the set $\{F,C_2^g\mid  g \neq 1\}$.

Finally, since $A \sim_{T_G} A_1 \oplus \cdots \oplus A_m$ and at least one of the components $A_i$ has quadratic codimension growth, the result follows.\end{proof}

The previous theorem shows that unitary $G$-graded varieties with quadratic growth are generated by a direct sum of algebras that generate minimal varieties. Recall that in \cite{Plamen}, the authors proved that varieties with at most linear codimension growth are also generated by the direct sum of algebras generating minimal varieties. This observation, combined with Theorem~\ref{quadraticgrowth}, leads to the following characterization.

\begin{corollary}
    Let $A$ be a unitary $G$-graded algebra. Then $c_n^G(A)\leq \alpha n^2 $ if and only if $A$ is $T_G$-equivalent to a finite direct sum of algebras generating minimal varieties with at most quadratic codimension growth.  
\end{corollary}

\vspace{0.5cm}
 \textbf{Acknowledgments}
\vspace{0.5cm}

The author is grateful to Prof. Ana Vieira for her constant support and valuable advice during the preparation of this paper, as well as to Filiphe Veiga for his valuable support throughout this work.

\end{document}